\documentclass[preprint, 12pt]{elsarticle}
\usepackage{verbatim}
\usepackage{amssymb}
\usepackage{amsmath}
\usepackage{amsthm}
\usepackage{enumerate}
\usepackage[active]{srcltx}
\usepackage{mathabx}
\usepackage{caption}
\usepackage{float}
\usepackage[utf8]{inputenc}
\usepackage{natbib}

 \usepackage[usenames,dvipsnames]{pstricks}
 \usepackage{epsfig}
\usepackage{pst-grad} 
\usepackage{pst-plot} 

\usepackage{geometry}

\newtheorem{theorem}{Theorem}[section]

\newtheorem{prop}[theorem]{Proposition}
\newtheorem{lemma}[theorem]{Lemma}

\newtheorem{corollary}[theorem]{Corollary}
\newtheorem{obs}[theorem]{Observation}

\newtheorem{prob}[theorem]{Problem}
\newtheorem{claim}{Claim}[theorem]

\newtheorem{definition}[theorem]{Definition}

\def\myheads#1;#2;{
\pagestyle{myheadings}
\markboth{{\sc\hfill #1\hfill\protect\makebox[0cm][r]{\rm\today}}}
{{\sc\protect\makebox[0cm][l]{\rm\today}\hfill #2\hfill}}
}

\newif\ifdeveloping

\developingtrue

\ifdeveloping
\fi

\newif\ifcommented
\newcommand{\chr}{\text{Chr}}
\newcommand{\vareps}{\varepsilon}
\newcommand{\ul}[1]{\underline{#1}}
\newcommand{\da}[1]{{#1}^\downarrow}
\newcommand{\ua}[1]{{#1}^\uparrow}
\newcommand{\comm}[1]{}

\ifcommented
\renewcommand{\comm}[1]{
\fbox{\fbox{\begin{minipage}{300pt}#1\end{minipage}}
}}

\fi

\newcommand{\smf}{\hspace{0.008 cm}^\smallfrown}

\newcommand{\mf}[1]{\mathfrak{#1}}
\newcommand{\uhp}{\upharpoonright}

\newcommand{\omg}{{\omega_1}}

\newcommand{\setm}{\setminus}

\newcommand{\subs}{\subset}
\newcommand{\dom}{\text{dom}}

\def\<{\left\langle}
\def\>{\right\rangle}
\def\br#1;#2;{\bigl[ {#1} \bigr]^ {#2} }

\newcommand{\oo}{{\omega}}

\begin{document}

\begin{frontmatter}

\title{Trees, ladders and graphs}
\author{D\'aniel T. Soukup}
\ead{daniel.soukup@mail.utoronto.ca}
\ead[url]{http://www.math.toronto.edu/$\sim$dsoukup}

\address{Department of Mathematics, 
University of Toronto, 
Bahen Centre 
40 St. George St., Room 6290
Toronto, Ontario 
CANADA 
M5S 2E4}

\date{\today}

\begin{abstract}
We introduce a new method to construct uncountably chromatic graphs from non special trees and ladder systems. Answering a question of P. Erd\H os and A. Hajnal from 1985, we construct graphs of chromatic number $\omg$ without uncountable $\omega$-connected subgraphs. Second, we build triangle free graphs of chromatic number $\omg$ without subgraphs isomorphic to $H_{\oo,\oo+2}$. 
\end{abstract}


\end{frontmatter}

\section{Introduction}

The \emph{chromatic number} of a graph $G$, denoted by $Chr(G)$, is the least (cardinal) number $\kappa$ such that the vertices of $G$ can be covered by $\kappa$ many independent sets. 
A fundamental problem of graph theory asks how large chromatic number affects structural properties of a graph and in particular, is it true that a graph with large chromatic number has certain obligatory subgraphs? 

The first result in this area is due to W. Tutte, alias Blanche Descartes \cite{tutte} (and independently A. Zykov \cite{zyk}): a construction of triangle free graphs of arbitrary large finite chromatic number. This result was significantly extended by Paul Erd\H os: there are graphs of arbitrary large finite chromatic number with arbitrary large girth (i.e. no small cycles) \cite{prob}.  Tutte's result extends to arbitrary \emph{infinite} chromatic graphs \cite{tfree} but can we generalize the second result?

In 1966 in their seminal paper, P. Erd\H os and Andr\' as Hajnal proved, to the great surprise of many, that every uncountably chromatic graph $G$ contains a copy of the complete bipartite graph $K_{n,\omega_1}$ for each $n\in \omega$  \cite{EH0} and hence a copy of a four cycle which is in striking contrast with the finite case and Tutte's construction. We are interested in the corollary that every uncountably chromatic graph $G$ contains an $n$-connected subgraph for each finite $n$. Erd\H os and Hajnal conjectured that this result can be extended to $\omega$-connectivity and asked in \cite{EH0} if every graph with chromatic number and size $\omg$ contains a subgraph of chromatic number $\omg$ which is $\omega$-connected i.e. any two points are connected by infinitely many pairwise disjoint paths. In 1985, Erd\H os and  Hajnal asked if every graph of chromatic number $\omg$ contains an uncountably chromatic $\omega$-connected subgraph \cite{EHsurv}; these problems were highly popularized and are included in the recent surveys \cite{kopesurv} and \cite{kopeerdos} as well. The main purpose of this paper is to solve the problem from \cite{EHsurv} using a new and rather flexible construction. 

Most advances on these questions are due to P\' eter Komj\'ath: in \cite{kopeconn}, he proves that every uncountably chromatic graph contains uncountably chromatic $n$-connected subgraphs (with minimal degree $\omega$) for each $n\in \omega$. Regarding the original Erd\H os-Hajnal question, he shows in \cite{kopecon1} that under PFA (the Proper Forcing Axiom) every graph of size and chromatic number $\omg$ contains an uncountably chromatic $\omega$-connected subgraph. 
 Furthermore, in the same paper, he presents a forcing which introduces a graph  of size and chromatic number $\omg$ \emph{without} uncountably chromatic $\omega$-connected subgraphs. His proof was slightly flawed (see the introduction of \cite{kopecon2} for details), but the error is corrected in the recent \cite{kopecon2} where he forces an example of a graph of size and chromatic number $\omg$ without \emph{uncountable} $\omega$-connected subgraphs. In particular, the original question from \cite{EH0} is independent of ZFC and the answer to the question from \cite{EHsurv} is consistently no.

Our main result is a complete answer to the above problem from \cite{EHsurv} in Theorem \ref{mainthm}: we construct a graph of chromatic number $\omg$ and size continuum which contains no uncountable $\omega$-connected subgraphs. Our proof is entirely within the usual axioms of set theory, i.e. we use no extra set theoretical assumptions nor forcing arguments. In Section \ref{consec2}, we refine our methods in order to get an uncountably chromatic graph $G$ such that any uncountable set $A$ contains two points which are connected by only finitely many pairwise disjoint paths \emph{even in $G$}.

In the second part of the paper, we present a different example using the same framework. In light of the fact that finite bipartite graphs embed into every uncountably chromatic graph, it is reasonable to ask if $K_{\omega,\omega}$ embeds into every uncountably chromatic graph. This is \emph{not} the case, however we have the following results by Hajnal and Komj\'ath \cite{half}: the graph $H_{\oo,\oo+1}$ embeds into every uncountably chromatic graph, however under CH (the Continuum Hypothesis) there is an uncountably chromatic graph without subgraphs isomorphic to $H_{\oo,\oo+2}$. Recall that $H_{\oo,\oo+2}$ is the graph on vertices $\{x_i,y_i,z,z':i\in \mathbb N\}$ with edges $$\bigl \{\{x_i,y_j\},\{x_i,z\},\{x_i,z'\}:i\leq j\in \mathbb N\bigr \}$$ and $H_{\oo,\oo+1}$ is the subgraph spanned by $\{x_i,y_i,z:i\in\mathbb N\}$.

In Section \ref{tfreesec}, we construct a graph of chromatic number $\omg$ and size continuum which is triangle free and has no subgraphs isomorphic to $H_{\oo,\oo+2}$. Thus the assumption of CH can be eliminated from the above result of Hajnal and Komj\'ath.  

Finally, in Section \ref{remarkssec} we close with some remarks on our constructions.

Our belief is that this framework for constructing uncountably chromatic graphs will find  diverse applications in infinite combinatorics.


\section{Preliminaries}

A set theoretic \emph{tree} $(T,\leq)$ is a partially ordered set such that $$t^\downarrow=\{s\in T:s< t\}$$ is well ordered for all $t\in T$. Note that this notion of a tree has little to do with  \emph{graph theoretic trees} i.e. connected graphs without circles; in this paper, by a tree we will always mean a set theoretic tree. Every tree admits a \emph{height} function: $ht(t)$ denotes the order type of $\da t$ for $t\in T$. The \emph{height of the tree T} is $\sup\{ht(t):t\in T\}$. 

 \begin{definition}
 Let $G(T)$ denote the \emph{comparability graph} of a tree $(T,\leq)$ i.e. the set of vertices of $G(T)$ is $T$ and $\{x,y\}\in [T]^2$ is an edge iff $x\leq y$ or $y\leq x$.
\end{definition}

Note that $A\subseteq T$ is independent in $G(T)$ iff it is an antichain in $T$ thus $Chr(G(T))\leq \omega$ iff $T$ is a special tree (i.e. the union of countably many antichains). We will be interested in constructing subgraphs of a graph $G(T)$ in order to solve the problems proposed in the introduction. 

We investigate connectivity properties of infinite graphs:

\begin{definition}We say that a set of vertices $F$ in a graph \emph{separates} two vertices $s$ and $t$ iff every path from $s$ to $t$ passes through $F$. We say that $F$ separates a set of vertices $A$ iff there are distinct $s,t\in A$ such that $F$ separates $s$ and $t$. \\
A graph $G$ is \emph{$\oo$-connected} iff no finite set separates two points of $G$. 
\end{definition}

Note that a subgraph $A$ of a graph $G$ is $\oo$-connected iff any two points $s,t\in A$ are connected by infinitely many pairwise disjoint paths in $A$.

\begin{obs}
Suppose that the tree $T$ does not contain uncountable chains. Then every uncountable subset of $G(T)$ can be separated by a countable set. 
\end{obs}
\begin{proof}
 Take any uncountable $A\subseteq T$ and find two incomparable elements $s,t\in A$. We claim that the countable set $\da s$ separates $s$ and $t$. Indeed, it is straightforward to see that the connected component of $s$ in $T\setm s^\downarrow$ contains only elements $r$ with $r\geq s$. 
\end{proof}

\begin{corollary}\label{treegraph}
 If $T$ is a non special tree with no uncountable chains then the graph $G=G(T)$ satisfies $Chr(G)>\oo$ while every uncountable set of vertices can be separated by a countable set.
\end{corollary}

The above simple corollary indicates that it is reasonable to investigate non special trees $T$ without uncountable chains and the corresponding graphs $G(T)$ regarding the Erd\H os-Hajnal question.

Classical examples of non special trees with no uncountable chains are 
\begin{enumerate}
\item $\sigma \mathbb Q=\{t\subs \mathbb Q:t$ is well ordered$\}$ with $s\leq t$ iff $s$ is an initial segment of $t$,
\item $T(S)=\{t\subs S:t$ is closed$\}$ with $s\leq t$ iff $s$ is an initial segment of $t$  where $S\subseteq \omg$ is stationary, co-stationary.
\end{enumerate}

The above trees both have size continuum and height $\omg$ and hence the corresponding compatibility graphs have chromatic number exactly $\omg$. $\sigma \mathbb Q$ was first studied by D. Kurepa in connection with Souslin's problem while the importance of trees of the form $T(S)$ was realized by S. Todorcevic \cite{stevocont}. We will use trees of the second kind i.e. of the form $T(S)$ throughout this paper. 

Trees of the form $T(S)$ have the following nice property: $T(S)$ has \emph{no branching at limit levels} i.e. $\da t = \da s$ implies $t=s$ for all limit elements $t,s\in T(S)$. This is proved using that every element of $T(S)$ is a \emph{closed} subset of $\omg$.

We will use the following notation: if $T=T(S)$ for some $S\subseteq \omg$ then let $$T_\delta=\{t\in T:\max (t)=\delta\}$$ for $\delta \in \omg$ and similarly define $T_{<\delta},T_{\leq \delta}$. 

In general, we use standard set theoretical notation and we refer the reader to \cite{kunen} for basic facts in set theory.

\section{Connectivity and chromatic number}\label{consec}

Our aim is to construct a subgraph $X$ of the comparability graph of $T(S)$ (where $S\subseteq \omg$ is stationary, costationary) which has chromatic number $\omg$ and does not contain uncountable \emph{$\omega$-connected} subsets i.e. every uncountable set of vertices $A$ contains two points $s,t\in A$ and a finite set $F\subs A$ such that any path $P\subs A$ between $s$ and $t$ passes through $F$.

\begin{definition} Suppose that $T$ is a tree. A \emph{ladder system on $T$} is a family  $\ul C=\{C_t:t\in T\}$ so that $C_t\subs t^\downarrow$ is either finite or a cofinal sequence of type $\omega$.
\end{definition}

Each ladder system $\ul C$ defines a subgraph $X_{\ul{C}}$ of $G(T)$ with vertices $T$ and edges $$\{\{s,t\}:s\in C_t,t\in T\}.$$ This is in direct analogy with the Hajnal-M\'at\'e graphs introduced in \cite{HM} i.e. the case where the tree $T$ is simply the cardinal $\omg$. We remark that it is independent of $ZFC$ whether there is a ladder system $\ul C$ on $\omg$ so that the corresponding Hajnal-M\'at\'e graph is uncountably chromatic. Also, ladder systems on trees were investigated in \cite{treeladder} in a different context.

\begin{definition} A ladder system $\ul C$ on $T$ is \emph{transitive} iff $$C_t\cap \da s\subseteq C_s$$ for all $t\in T$ and $s\in C_t$.
\end{definition}

Note that $\ul{C}$ is transitive iff $C_t$ spans a complete graph in $X_{\ul{C}}$ for all $t\in T$. The next two lemmas explain why we introduced the notion of a transitive ladder system. 

\begin{lemma}\label{mon} Suppose that $T$ is a tree and $\ul C$ is a  transitive ladder system on $T$. If $s,t\in T$  and $P$ is a finite path in $X_{\ul C}$ from $s$ to $t$ then there is a path $Q\subseteq P$ which is the union of two monotone paths.
\end{lemma}

A \emph{monotone path} in $X_{\ul C}$ is a path which is a chain in the tree ordering.

\begin{proof} Let $Q\subseteq P$ be a path of \emph{minimal length} from $s$ to $t$. Let $\{q_i:i<n\}$ enumerate $Q$ by its path ordering. Note that we cannot have $q_{i-1},q_{i+1}<q_i$ for any $1\leq i\leq n-1$; indeed, this would imply that $q_{i-1},q_{i+1}\in C_{q_i}$ and hence, by transitivity, $q_{i-1}$ and $q_{i+1}$ are connected by an edge which contradicts the minimality of $Q$. This means that if $q_{i-1}<q_i$ then $q_i< q_{i+1}$ for any $1\leq i\leq n-1$ i.e. $Q$ is monotone increasing from the first step up (in the tree ordering). In other words, $Q$ is the union of a monotone decreasing and a monotone increasing path.

\end{proof}

\begin{lemma}\label{con} Suppose that $T$ is a tree and  $\ul C$ is a  transitive ladder system on $T$. If $T$ has no branching at limit levels and contains no uncountable chains then $X_{\ul C}$ contains no uncountable $\omega$-connected subsets.
\end{lemma}
\begin{proof} Fix an uncountable $A\subseteq T$ and find two incomparable elements $s,t\in A$. Let $r$ denote the maximal common initial part of $s$ and $t$; this exists and $r<s,t$ as $T$ does not branch at limits. Find $s'\in A$ so that $r< s'\leq s$ and $$s'^\downarrow \cap A\subseteq r^\downarrow \cup \{r\}.$$ We claim that $s'$ and $t$ are separated by the finite set $F=\{r\}\cup (r^\downarrow\cap C_{s'})$.

Suppose that $P=\{p_i:i<n\}\subs A$ is a finite path from $p_0=s'$ to $t$. By Lemma \ref{mon}, we can suppose that $P$ is the union of two monotone paths. Note that $p_1\in A\cap C_{s'}$ as $p_1<p_0=s'$ and hence $p_1 \in r^\downarrow \cup \{r\}$ by the choice of $s'$. In particular, $p_1\in F$; this shows that $F$ separates $s'$ and $t$.
\end{proof}

We are ready now to prove our main result:

\begin{theorem}\label{mainthm} Fix a stationary, costationary $S\subs \omg$ and let $T=T(S)$. Then there is a subgraph $X$ of $G(T)$ such that $Chr(X)=\omg$ and $X$ contains no uncountable $\omega$-connected subsets.
\end{theorem}
\begin{proof} It suffices to show that there is a transitive ladder system $\ul C$ on $T$ such that $Chr(X_{\ul C})=\omg$; indeed, $T$ has no uncountable chains nor branches at limit levels hence Lemma \ref{con} implies that $X_{\ul C}$ has no uncountable $\omega$-connected subsets. Furthermore, as the tree $T$ has height $\omg$ we have that $Chr(X_{\ul C})\leq Chr(G(T))\leq \omg$ for any $\ul C$. Thus we will have to show that $Chr(X_{\ul C})>\omega$ in the end.

By induction on $\delta\in S'$ (where $S'$ denotes the accumulation points of $S$), we define the transitive ladder system $\ul C$ on $T_{<\delta}$ and hence the corresponding part of $X_{\ul C}$ on $T_{<\delta}$.

Fix $\delta\in S'$ and suppose that we already defined a transitive ladder system $(C_t:t\in T_{<\delta})$ on $T_{<\delta}$. We extend this to $T_{<\delta^+}$ while preserving transitivity where $\delta^+$ is the minimum of $S'\setm (\delta+1)$. Note that transitivity is necessarily preserved at limit steps. If $\delta\notin S$ then we let $C_t=\emptyset$ for $t\in T_{<\delta^+}\setm T_{<\delta}$. 

Suppose that $\delta\in S$ and let us define $C_t$ for $t\in T_{\delta}$ first. Let $\{(A_\xi,f_\xi):\xi<\mf c\}$ denote a 1-1 enumeration of all the pairs $(A,f)$ so that $A\in [T_{<\delta}]^{\omega}$, $f:A\to \omega$ and $A$ satisfies

\begin{enumerate}[$(\star)$]
\item for every $t\in A$ and $\vareps<\delta$ there are incomparable $s^0,s^1\in A$ so that $s^i\geq t$ and  $\max(s^i)> \vareps$ for $i<2$.
\end{enumerate}

By induction on $\xi<\mf c$ we will find $t_\xi\in T_\delta\setm \{t_\zeta:\zeta<\xi\}$ and sets $C_{t_\xi}\subseteq \da t_\xi$ so that $$\bigl (C_t:t\in T_{<\delta}\cup \{t_\xi:\xi<\mf c\}\bigr )$$ is still a transitive ladder system. We will let $C_t=\emptyset$ for $t\in T_{\delta}\setm \{t_\xi:\xi<\mf c\}$. 

Fix a cofinal $\omega$-type sequence $\{\delta_n:n\in\omega\}$ in $\delta$. Suppose we defined $t_\zeta\in T_\delta$ and $C_{t_\zeta}\subseteq \da t_\zeta$ for $\zeta<\xi$.

Define a map $\psi:2^{<\omega}\to A_\xi$ and a partial map $\varphi:2^{<\omega}\to A_\xi$ so that 

\begin{enumerate}[(i)]

\item $\psi$ and $\varphi$ are order preserving injections and $$\psi(x)\leq \varphi(x)\leq\psi(x\smf i)$$ for $i<2$ provided that $x\in \dom(\varphi)$,
\item \label{transcond} $\{\varphi(x\uhp k):k\leq |x|, x\uhp k\in \dom(\varphi)\}$ is a complete graph in $T_{<\delta}$,
\item \label{branchcond} $\psi(x\smf 0)$ and $\psi(x\smf 1)$ are incomparable and contained in $A_\xi\setm T_{<\delta_n}$,
\item \label{cond} \underline{if} there is $t\in A_\xi$ such that $t\geq \psi(x)$, $f_\xi(t)=n$ and $\{\varphi(x\uhp k):k< |x|, x\uhp k\in \dom(\varphi)\}\cup\{t\}$ is complete \underline{then} $x\in \dom(\varphi)$ and $f_\xi(\varphi(x))=n$ as well 
\end{enumerate}
 for all $n\in \omega$ and $x\in 2^n$.

We define $\psi(x)$ and $\varphi(x)$ for $x\in 2^n$ by induction on $n\in \omega$. We set $\psi(\emptyset)\in A_\xi$ arbitrarily. If $\psi(x)$ is defined for some $x\in 2^n$ then check if $$R_x=\{t\in A_\xi:t\geq \psi(x), f_\xi(t)=n \text{ and }\{\varphi(x\uhp k):k< |x|, x\uhp k\in \dom(\varphi)\}\cup\{t\} \text{  is complete}\}$$
 is empty or not. If $R_x\neq \emptyset$ then we put $x$ into $\dom(\varphi)$ and pick $\varphi(x)$ from $R_x$ arbitrarily; otherwise $x\notin \dom(\varphi)$.
 Now find incomparable $\psi(x\smf 0),\psi(x\smf 1)\in A_\xi$ above $\psi(x)$ so that $\max(\psi(x\smf i))\geq \delta_n$ and $\psi(x\smf i)\geq \varphi(x)$ if $x\in \dom(\varphi)$; this can be done as $A_\xi$ satisfies (1) above. This finishes the construction of $\psi$ and $\varphi$.

\vspace{0.5 cm}

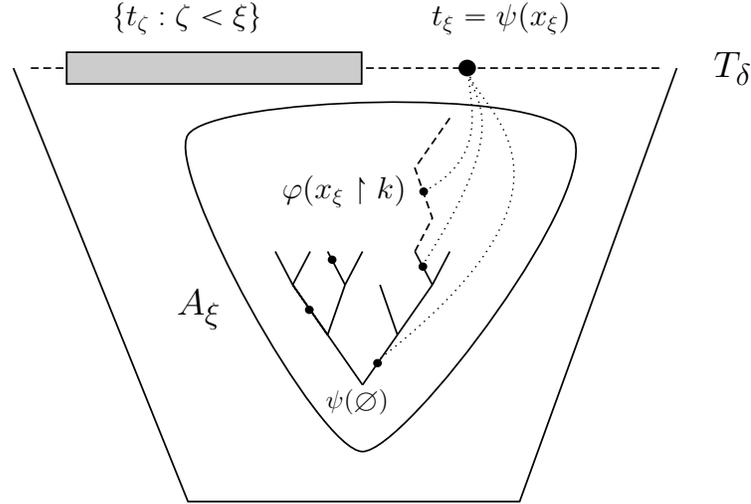
\begin{figure}[H]
\centering

\resizebox{9 cm}{6 cm}{
\begin{pspicture}(0,-5.410011)(15.237574,5.410011)

\definecolor{colour0}{rgb}{0.8,0.8,0.8}
\psline[linecolor=black, linewidth=0.04](0.01867096,5.010011)(4.018671,-5.389989)(11.618671,-5.389989)(15.218671,5.010011)(15.218671,5.010011)

\psline[linecolor=black, linewidth=0.04, linestyle=dashed, dash=0.17638889cm 0.10583334cm](0.41867095,5.010011)(14.818671,5.010011)(14.818671,5.010011)
\psframe[linecolor=black, linewidth=0.04, fillstyle=solid,fillcolor=colour0, dimen=outer](8.018671,5.410011)(1.218671,4.610011)
\psbezier[linecolor=black, linewidth=0.04](4.018671,3.410011)(4.573371,4.288286)(12.018671,4.610011)(12.818671,3.410011)(13.618671,2.210011)(9.018671,-4.189989)(8.018671,-4.189989)(7.018671,-4.189989)(3.4639707,2.5317357)(4.018671,3.410011)
\psline[linecolor=black, linewidth=0.04](8.018671,-2.589989)(7.218671,-1.389989)(7.218671,-1.389989)
\psline[linecolor=black, linewidth=0.04](8.018671,-2.589989)(8.818671,-1.389989)
\psline[linecolor=black, linewidth=0.04](7.218671,-1.389989)(6.418671,-0.18998902)(7.218671,-1.389989)
\psline[linecolor=black, linewidth=0.04](7.218671,-1.389989)(7.618671,-0.18998902)
\psline[linecolor=black, linewidth=0.04](8.818671,-1.389989)(8.418671,-0.18998902)
\psline[linecolor=black, linewidth=0.04](8.818671,-1.389989)(9.618671,-0.18998902)
\psline[linecolor=black, linewidth=0.04](6.418671,-0.18998902)(6.018671,0.610011)(6.018671,0.610011)
\psline[linecolor=black, linewidth=0.04](6.418671,-0.18998902)(6.8186707,0.610011)
\psline[linecolor=black, linewidth=0.04](7.618671,-0.18998902)(7.218671,0.610011)(7.218671,0.610011)
\psline[linecolor=black, linewidth=0.04](7.618671,-0.18998902)(8.018671,0.610011)
\psline[linecolor=black, linewidth=0.04](9.618671,-0.18998902)(9.218671,0.610011)(9.218671,0.610011)
\psline[linecolor=black, linewidth=0.04](9.618671,-0.18998902)(10.018671,0.610011)
\psline[linecolor=black, linewidth=0.04, linestyle=dashed, dash=0.17638889cm 0.10583334cm](9.218671,0.610011)(9.618671,1.4100109)(9.218671,2.610011)(10.018671,3.810011)(10.018671,3.810011)
\psdots[linecolor=black, dotsize=0.4](10.418671,5.010011)
\psdots[linecolor=black, dotsize=0.2](8.358671,-2.069989)
\psdots[linecolor=black, dotsize=0.2](9.398671,0.250011)
\psdots[linecolor=black, dotsize=0.2](9.418671,2.050011)
\psdots[linecolor=black, dotsize=0.2](6.7999997,-0.789989)
\psdots[linecolor=black, dotsize=0.2](7.32,0.410011)
\psbezier[linecolor=black, linewidth=0.04, linestyle=dotted, dotsep=0.10583334cm](10.438671,5.010011)(10.438671,4.210011)(10.978671,2.950011)(9.458671,2.070011)
\psbezier[linecolor=black, linewidth=0.04, linestyle=dotted, dotsep=0.10583334cm](10.438671,5.030011)(10.438671,4.230011)(11.638671,3.910011)(9.398671,0.290011)
\psbezier[linecolor=black, linewidth=0.04, linestyle=dotted, dotsep=0.10583334cm](10.458671,4.930011)(10.458671,4.130011)(13.838671,2.430011)(8.358671,-2.009989)
\rput(4.2,-0.789989){ \Huge{$A_\xi$}}
\rput(4,6.2){\LARGE{$\{t_\zeta:\zeta<\xi\}$}}
\rput(11.2,6.2){\LARGE{$t_\xi=\psi(x_\xi)$}}
\rput(16.5,5){\Huge{$T_\delta$}}
\rput(7.6,2.050011){\LARGE{$\varphi(x_\xi\uhp k)$}}
\rput(7.9,-3){\Large{$\psi(\emptyset)$}}
\end{pspicture}
} \caption{ Step $\xi$ in the induction.}
\end{figure}

We extend $\psi$ to $2^\omega$ in the obvious way:

$$\psi(x)=\bigcup\{\psi(x\uhp k):k<\oo\}\cup \{\delta\}$$

for $x\in 2^\oo$; note that $\psi(x)$ is a closed subset of $S$ by the second part of condition (\ref{branchcond}) and hence $\psi(x) \in T_{\delta}$ for all $x\in 2^\oo$. Also, $\psi$ remains 1-1 on $2^\oo$ by the first part of condition (\ref{branchcond}). Hence, we can find an $x_\xi \in 2^\omega$ such that $\psi(x_\xi)\in T_{\delta}\setm \{t_\zeta:\zeta<\xi\}$ and we let $t_\xi=\psi(x_\xi)$. Finally, let $$C_{t_\xi}=\bigl \{\varphi(x_\xi\uhp k):k<\oo, x_\xi\uhp k\in \dom(\varphi)\bigr \}.$$ 

Transitivity of this extension is assured by condition (\ref{transcond}). This finishes the induction on $\xi<\mf c$ and we have a transitive ladder system $\ul C=(C_t:t\in T_{\leq \delta})$ on $T_{\leq \delta}$. We now simply let $C_t=\emptyset$ for $t\in T_{<\delta^+}\setm T_{\leq\delta}$. This finishes step $\delta$ of the main induction and hence, in the end, we have a transitive ladder system $\ul C$ on $T$.

 We are left to prove:

\begin{claim}$Chr(X_{\ul C})>\omega$.
\end{claim}
\begin{proof} Fix a colouring $f:T\to \omega$; we will find $s,t\in T$ so that $f(s)=f(t)$ and $s\in C_t$. Take a countable elementary submodel $M\prec H(\omega_2)$ (where $H(\omega_2)$ is the collection of sets with hereditary cardinality $\leq\omg$) so that $S,\ul C, f\in M$ and $\delta=M\cap \omg\in S$; this can be done as $S$ is stationary. 

Consider the construction of $\{C_t:t\in T_{\delta}\}$. If we set $A=M\cap T$ then there must be a $\xi<\mf c$ so that $(A,f\uhp A)=(A_\xi,f_\xi)$; $M$ being an elementary submodel ensures that $A$ satisfies property (1) as the branching of $T$ reflects to $M$.

Our goal now is to prove that there is an $s\in C_{t_\xi}$ so that $f(s)=f(t_\xi)$; we let $n=f(t_\xi)$. Recall that in the definition of $t_\xi$ we had two maps $\psi$ and $\varphi$ and $t_\xi$ was of the form $\psi(x_\xi)$ for an $x_\xi\in 2^\oo$. $C_{t_\xi}$ was defined to be $\{\varphi(x_\xi\uhp k):k<\oo, x_\xi\uhp k\in \dom(\varphi)\}$. 

Now, recall the definition of $\varphi(x_\xi\uhp n)$: we looked at the set 
\begin{multline}
R_{x_\xi\uhp n}=\{s\in A_\xi:s\geq \psi(x_\xi\uhp n), f_\xi(s)=n \text{ and }\\ 
\{\varphi(x_\xi\uhp k):k< n, x_\xi \uhp k\in \dom(\varphi)\}\cup\{s\} \text{  is complete}\}
\end{multline}

and if $R_{x_\xi\uhp n}$ was not empty the we chose $\varphi(x_\xi\uhp n)\in R_{x_\xi\uhp n}$; in particular, $f(s)=n$ for $s=\varphi(x_\xi\uhp n) \in C_{t_\xi}$ which would finish the proof.

Let us show that  $R_{x_\xi\uhp n}$ is not empty. Let
\begin{multline}
R=\{s\in T:s\geq \psi(x_\xi\uhp n), f(s)=n \text{ and }\\ 
\{\varphi(x_\xi\uhp k):k< n, x_\xi \uhp k\in \dom(\varphi)\}\cup\{s\} \text{  is complete}\}
\end{multline}
and note that $R$ is in the model $M$ and $R_{x_\xi\uhp n}=R\cap M$. Hence, by elementarity, it suffices to show that $R\neq \emptyset$. However, this is clear as $t_\xi\in R$.

\end{proof}

This finishes the proof of the theorem.
\end{proof}

\section{A highly disconnected variation}\label{consec2}

In Theorem \ref{mainthm}, we constructed a graph $X$ such that any uncountable set $A$ \emph{contained} two incomparable points $s,t$ which are \emph{separated in A} by a finite set $F$; in particular, there could be paths connecting $s$ and $t$ which avoid $F$ by leaving $A$. The aim of this section is to refine the methods of Section \ref{consec} and produce a ladder system $\ul C$ on a tree $T$ such that \emph{any two $<_T$-incomparable vertices} are separated by a finite set in the graph $X_{\ul C}$ i.e. if $s,t\in T$ are incomparable then there is a finite set $F\subs T$ such that \emph{any path $P$} from $s$ to $t$ intersects $F$. Note that this separation property is stronger than the lack of uncountable $\omega$-connected sets.

We use the following notation: if $T$ is a tree then
\begin{itemize}
	\item let $supp(\ul C)=\{t\in T: |C_t|=\omega\}$ for any ladder system $\ul C$ on $T$,
	\item if $s\in T$ and $\varepsilon < ht(s)$ then $s\uhp \varepsilon$ denotes the unique element $r\in \da s$ with $ht(r)=\varepsilon$.
\end{itemize}

Let us introduce a somewhat technical property of ladder systems on trees. First, we need the following definition: we will say that a sequence $\ul \eta=(\eta_t:t\in T)$ is a \emph{true ladder system} on a tree $T$ iff $\eta_t=\{t\}$ for all successor $t\in T$ and $\eta_t$ is a cofinal sequence of type $\omega$ in $\da t$ if $t$ is limit. 

\begin{definition}
Suppose that $\ul C$ is a ladder systems on a tree $T$. We say that $\ul C$ is \emph{coherent} iff
\begin{enumerate}
		\item $C_s=C_t\cap \da s$ 
	\end{enumerate}
	for all $t\in supp(\ul C)$ and $s\in C_t$ with $|C_s|<\oo$ and there is a true ladder system $\ul \eta$ on $T$ such that
 \begin{enumerate}
\setcounter{enumi}{1}
 \item  $\eta_t\cap \da s \sqsubseteq \eta_s$,
\item $C_t\cap \da r = C_s \cap \da r$ for $\displaystyle{r=s\uhp ht(\max_{<_T}(\eta_t\cap \da s))+1}$
\end{enumerate}
  for every $s,t\in supp(\ul C)$ with $s\in C_t$.
	
\end{definition}

The next lemma explains how transitivity and coherence of a ladder system $\ul C$ gives the desired separation property of the graph $X_{\ul C}$.

\begin{lemma}\label{etacoh}
 Suppose that $T$ is a tree with no branching at limits. If $\ul C$ is a transitive and coherent ladder system on $T$ then any two $<_T$-incomparable points are separated by a finite set in $X_{\ul C}$.

\end{lemma}

We use the following notation in the proof: if $s,t\in T$ then $\Delta(s,t)$ denotes the maximal common initial segment of $s$ and $t$. 

\begin{proof} Fix a point $t'\in T$ and we prove that for any $t\in T$ which is incomparable with $t'$ there is a finite set $F_t$ which separates $t$ and $t'$.

We will actually prove that the following choice of $F_t$ works: let $F_t=C_t$ if $t\notin supp(C)$ and let $$F_t=C_t\cap \da r_t \text{ where } r_t=t\uhp ht(\min_{<_T}\{r\in \eta_t:r>\Delta(t,t')\})+1$$ if $t\in supp(\ul C)$.  The proof is done by induction on $ht(t)$. 

It is clear that if $t\notin supp(\ul C)$ then $F_t$ separates $t$ and $t'$ by transitivity and Lemma \ref{mon}.

Suppose that $t\in supp(\ul C)$; we note that $r_t<t$ as $\eta_t$ is nontrivial and $\Delta(t,t')<t$ and hence $F_t$ is finite. Now, suppose that $P=\{p_0\dots p_n\}$ is a path from $t$ to $t'$ which avoids $F_t$; we can suppose that $P$ has minimal length and is the union of two monotone paths by Lemma \ref{mon}. We have $p_0=t>p_1\in C_t$ and $p_1>\Delta(t,t')$ and hence $\Delta(t,t')=\Delta(p_1,t')$. We let $s=p_1$ and note that $p_2\in C_s\setm C_t$ as $P$ has minimal length and $\ul C$ is transitive. If $|C_s|<\omega$ then by assumption $C_t\cap \da s= C_s$ which contradicts $p_2\in C_s\setm C_t$.

We conclude that $s\in supp(\ul C)$. As $F_s$ separates $s$ and $t'$ (by the inductive hypothesis) it suffices to prove that $F_s=F_t$. By definition $$r_s=s\uhp ht(\min_{<_T}\{r\in \eta_s:r>\Delta(s,t')\})+1.$$ As $t,s\in supp(\ul C)$ we have $\eta_t\cap \da s\sqsubseteq \eta_s$ by coherence. Hence, by $r_t<s$ and $\Delta(t,t')=\Delta(s,t')$ we must have $r_s=r_t$. Finally, by coherence again, we have $F_t=C_t\cap \da{r_t}=C_s\cap \da{r_s}=F_s$.
\end{proof}

We are ready to prove the main result of this section.

\begin{theorem} \label{mainthm2}
 Fix a stationary, costationary $S\subs \omega_1$ and let $T=T(S)$. Then there is a subgraph $X$ of $G(T)$ such that $Chr(X)=\omg$ and any two $<_T$-incomparable points are separated by a finite set in $X$. In particular, every uncountable set $A\subseteq T$ contains two vertices which are separated by a finite set in $X$.
\end{theorem}

\begin{proof}
It suffices to find a transitive and coherent ladder system $\ul C$ on $T$ such that $Chr(X_{\ul C})>\oo$; indeed, $T$ does not branch at limits (and every uncountable $A\subseteq T$ contains two incomparable points) hence Lemma \ref{etacoh} can be applied.

We define a true ladder system $\ul \eta$ on $T$ first: fix a true ladder system $\{\nu_\delta:\delta\in \omg\}$ on $\omg$ and let $$\eta_t=\{t\cap (\vareps+1):\vareps\in \nu_\delta\} \text{ where } \delta=\max(t)$$ for any limit $t\in T$ and let $\eta_t=\{t\}$ for any successor $t\in T$.

 Let $D=(S\cap S')'$ and let $\delta^+$ denote $\min D\setm (\delta+1)$ for $\delta\in D$. By induction on $\delta\in D$, we define a transitive ladder system $\ul C$ on $T_{<\delta}$, and hence the corresponding graph on $T_{<\delta}$, so that $\ul C$ is coherent and its coherence is witnessed by $\ul \eta$.

Suppose we defined $(C_t)_{t\in T_{<\delta}}$ for some $\delta\in D$. We define $\ul C$ on $T_{<\delta^+}$ in two steps: first we define $C_t$ for $t\in T_{\delta}$ so that $(C_t)_{t\in T_{\leq \delta}}$ is still transitive and coherent and then extend to $T_{<\delta^+}$ in the trivial way: we let $C_t=\emptyset$ for $t\in T_{<\delta^+}\setm T_{\leq \delta}$. 

We can suppose $\delta \in S$ otherwise $T_\delta=\emptyset$. Let $\{((A^\xi_{ n})_{n\in \omega}, f_\xi):\xi<\mf c\}$ denote a 1-1 enumeration of all pairs  $((A_n)_{n\in \omega}, f)$  where 
\begin{enumerate}
\item $A_n\subs A_{n+1}\in [T_{<\delta}]^{\oo}$, 
\item\label{br2} for every $t\in A_n$ and $\vareps<\delta_n=\sup\{\max(s):s\in A_n\}$ there are incomparable $s^0,s^1\in A_n$ so that $s^i\geq t$  and $\max(s^i)> \vareps$ for $i<2$,

\item $f_\xi:A\to \omega$ where $A=\bigcup \{A_n:n\in\omega\}$,
\item $(\delta_n)_{n\in \oo}$ is a strictly increasing cofinal sequence in $\delta$.
\end{enumerate}

By induction on $\xi<\mf c$, we define $t_\xi\in T_{\delta}\setm\{t_\zeta:\zeta<\xi\}$ and $C_{t_\xi}\subseteq \da t_\xi$ while preserving transitivity and coherence. Suppose we defined $t_\zeta\in T_\delta$ and $C_{t_\zeta}\subseteq \da t_\zeta$ for $\zeta<\xi$.

We let $A_\xi =\bigcup \{A^\xi_{n}:n\in \omega\}$ and $\delta^\xi_{ n}=\sup\{\max(s):t\in A^\xi_{ n}\}$. Also, let $$\vareps_n=\max (\{\delta^\xi_n\}\cup (\nu_\delta\cap \delta^\xi_{n+1}))$$ for $n\in \oo$ and $\vareps_{-1}=\max(\nu_\delta\cap \delta^\xi_0)$. Finally, let $\{l_n:n\in \oo\}$ enumerate each natural number infinitely many times.

Define a map $\psi:2^{<\omega}\to A_\xi$ and a partial map $\varphi:2^{<\omega}\to A_\xi$ so that

\begin{enumerate}[(i)]

\item \label{cond1} $\psi$ and $\varphi$ are order preserving injections and $$\psi(x)\leq \varphi(x)\leq \psi(x\smf i)$$ for $i<2$ provided that $x\in \dom(\varphi)$,
\item \label{transcond2} $\{\varphi(x\uhp k):k\leq |x|, x\uhp k\in \dom(\varphi)\}$ is a complete graph in $T_{<\delta}$,
\item \label{branchcond2} $\psi(x\smf 0)$ and $\psi(x\smf 1)$ are incomparable and contained in $A^\xi_{n+1}\setm T_{<\varepsilon_n}$,
\item  \underline{if} there is an $s\in A^\xi_n$ such that
\begin{enumerate}
	\item $s\geq \psi(x)$, $f_\xi(s)=l_n$,
	\item $\{\varphi(x\uhp k):k< |x|, x\uhp k\in \dom(\varphi)\}\cup\{s\}$ is complete,
\item \label{infcond} $\nu_\delta\cap \varepsilon_{n-1}\sqsubseteq \nu_{\max(s)}$ and
 $$C_s\cap \da r= \{\varphi(x\uhp k):k< n, x\uhp k\in \dom(\varphi)\}$$ for $r=s\cap (\vareps_{n-1}+1)$ if $C_s$ is infinite,
\item \label{finite} $C_s=\{\varphi(x\uhp k):k< n, x\uhp k\in \dom(\varphi)\}$ if $C_s$ is finite
\end{enumerate} 
\underline{then} $x\in \dom(\varphi)$ and $s=\varphi(x)$ satisfies (a)-(d) as well,
\item if (iv) fails then $x\notin \dom(\varphi)$
\end{enumerate}
 for all $n\in \omega$ and $x\in 2^n$.

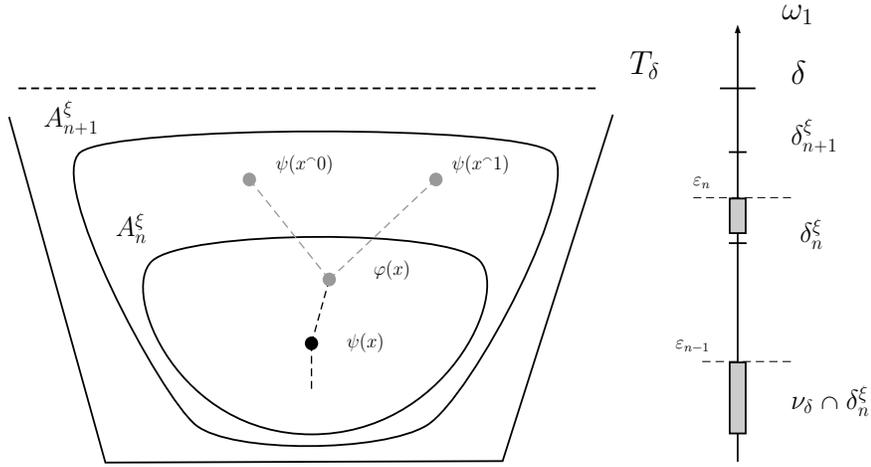
\begin{figure}[H]
\centering

\resizebox{13 cm}{6 cm}{
\begin{pspicture}(0,-4.9526877)(21.751286,4.9526877)
\definecolor{colour0}{rgb}{0.6,0.6,0.6}
\definecolor{colour1}{rgb}{0.8,0.8,0.8}
\psline[linecolor=black, linewidth=0.03, linestyle=dashed, dash=0.17638889cm 0.10583334cm](6.831285,-2.3133376)(7.2312846,-0.9133375)
\psline[linecolor=black, linewidth=0.04](0.019225586,2.6594565)(2.188918,-4.932619)(11.136232,-4.8997703)(13.6191435,2.709386)(13.6191435,2.709386)
\psline[linecolor=black, linewidth=0.04, arrowsize=0.05291666666666667cm 2.0,arrowlength=1.4,arrowinset=0.0]{<-}(16.431284,4.6866627)(16.431284,-4.9133377)(16.431284,-4.9133377)
\psline[linecolor=black, linewidth=0.04](16.031284,3.2866623)(16.831285,3.2866623)(16.831285,3.2866623)
\psline[linecolor=black, linewidth=0.04](16.231285,1.8866625)(16.631285,1.8866625)(16.631285,1.8866625)
\psline[linecolor=black, linewidth=0.04](16.231285,-0.113337554)(16.631285,-0.113337554)(16.631285,-0.113337554)
\psline[linecolor=black, linewidth=0.04, linestyle=dashed, dash=0.17638889cm 0.10583334cm](13.231285,3.2866623)(0.2312848,3.2866623)

\psbezier[linecolor=black, linewidth=0.04](1.6312848,1.8866625)(2.4514117,2.4779167)(11.431285,2.5066624)(12.231285,1.8866625)(13.031284,1.2666625)(10.431285,-3.5133376)(9.431285,-4.1133375)(8.431285,-4.7133374)(5.031285,-4.7333374)(4.2312846,-4.1133375)(3.431285,-3.4933376)(0.8111578,1.2954081)(1.6312848,1.8866625)
\psbezier[linecolor=black, linewidth=0.04](3.2312849,-0.51333755)(3.9562843,0.17541191)(9.96449,0.23190376)(10.631285,-0.51333755)(11.2980795,-1.2585789)(9.8312845,-4.3133373)(6.831285,-4.3133373)(3.8312848,-4.3133373)(2.5062854,-1.202087)(3.2312849,-0.51333755)

\psdots[linecolor=black, dotsize=0.3](6.831285,-2.3133376)
\psdots[linecolor=colour0, dotsize=0.3](7.2312846,-0.9133375)
\psdots[linecolor=colour0, dotsize=0.3](5.431285,1.2866625)
\psdots[linecolor=colour0, dotsize=0.3](9.631285,1.2866625)
\psframe[linecolor=black, linewidth=0.04, fillstyle=solid,fillcolor=colour1, dimen=outer](16.631285,0.8866624)(16.231285,0.08666244)
\psframe[linecolor=black, linewidth=0.04, fillstyle=solid,fillcolor=colour1, dimen=outer](16.631285,-2.7133377)(16.231285,-4.3133373)
\psline[linecolor=colour0, linewidth=0.03, linestyle=dashed, dash=0.17638889cm 0.10583334cm](7.2312846,-0.9133375)(5.431285,1.2866625)
\psline[linecolor=colour0, linewidth=0.03, linestyle=dashed, dash=0.17638889cm 0.10583334cm](7.2312846,-0.9133375)(9.631285,1.2866625)
\psline[linecolor=black, linewidth=0.03, linestyle=dashed, dash=0.17638889cm 0.10583334cm](6.831285,-2.3133376)(6.831285,-3.3133376)

\psline[linecolor=black, linewidth=0.02, linestyle=dashed, dash=0.17638889cm 0.10583334cm](15.631285,-2.7133377)(17.631285,-2.7133377)
\psline[linecolor=black, linewidth=0.02, linestyle=dashed, dash=0.17638889cm 0.10583334cm](15.431285,0.8866624)(17.631285,0.8866624)

\rput[bl](0.8,2.25){\Large{$A^\xi_{n+1}$}}
\rput[bl](2.42,-0.113337554){\Large{$A^\xi_n$}}
\rput[bl](15.031284,-2.5133376){$\varepsilon_{n-1}$}
\rput[bl](15.431285,1.0866624){$\varepsilon_n$}
\rput[bl](14,3.5){\LARGE{$T_\delta$}}
\rput[bl](6.031285,1.4){$\psi(x\smf 0)$}
\rput[bl](10,1.4){$\psi(x\smf 1)$}
\rput[bl](7.6312847,-2.5133376){$\psi(x)$}
\rput[bl](8.231285,-0.9133375){$\varphi(x)$}
\rput[bl](17.631285,-3.9133375){\Large{$\nu_\delta\cap \delta^\xi_n$}}
\rput[bl](17.431284,4.6866627){\LARGE{$\omg$}}
\rput[bl](17.631285,3.3866624){\LARGE{$\delta$}}
\rput[bl](17.831285,-0.25){\Large{$\delta^\xi_n$}}
\rput[bl](17.631285,1.86){\Large{$\delta^\xi_{n+1}$}}
\end{pspicture}
}

\caption{Extending the maps $\varphi$ and $\psi$.}%
\label{}%
\end{figure}

We define $\psi(x)$ and $\varphi(x)$ for $x\in 2^n$ by induction on $n\in \omega$. We let $\psi(\emptyset)\in A^\xi_0\setm T_{\leq \vareps_{-1}}$ arbitrarily. Given $\psi(x)$, we consider the set $R^\xi_x$ of all elements $s\in A^\xi_n$ such that
\begin{enumerate}[(a')]
	\item $s\geq \psi(x)$, $f_\xi(s)=l_n$,
	\item $\{\varphi(x\uhp k):k< |x|, x\uhp k\in \dom(\varphi)\}\cup\{s\}$ is complete,
\item  $\nu_\delta\cap \varepsilon_{n-1}\sqsubseteq \nu_{\max(s)}$ and
 $$C_s\cap \da r= \{\varphi(x\uhp k):k< n, x\uhp k\in \dom(\varphi)\}$$ for $r=s\cap (\vareps_{n-1}+1)$ if $C_s$ is infinite,
\item $C_s=\{\varphi(x\uhp k):k< n, x\uhp k\in \dom(\varphi)\}$ if $C_s$ is finite.
\end{enumerate} 

If $R_x^\xi$ is not empty then we set $x\in \dom(\varphi)$ and choose an arbitrary $\varphi(x)\in R_x^\xi$. Otherwise, $x\notin \dom(\varphi)$. Now, we simply pick $\psi(x\smf i)$ for $i<2$ satisfying conditions (\ref{cond1}) and (\ref{branchcond2}) by applying condition (\ref{br2}) for $A^\xi_{n+1}$. This finishes the construction of $\psi$ and $\varphi$.

We extend $\psi$ to $2^\omega$ in the obvious way:

$$\psi(x)=\bigcup\{\psi(x\uhp k):k<\oo\}\cup \{\delta\}$$

for $x\in 2^\oo$; note that $\psi(x)$ is a closed subset of $S$ by the second part of condition (\ref{branchcond}) hence $\psi(x) \in T_{\delta}$ for all $x\in 2^\oo$ . Also, $\psi$ remains 1-1 on $2^\oo$ by the first part of condition (\ref{branchcond}). Hence, we can find an $x_\xi \in 2^\omega$ such that $\psi(x_\xi)\in T_{\delta}\setm \{t_\zeta:\zeta<\xi\}$ and we let $t_\xi=\psi(x_\xi)$. Finally, let $$C_{t_\xi}=\bigl \{\varphi(x_\xi\uhp k):k<\oo, x_\xi\uhp k\in \dom(\varphi)\bigr \}.$$  This finishes the induction on $\xi<\mf c$ and we define $C_t=\emptyset$ for $t\in T_\delta\setm \{t_\xi:\xi<\mf c\}$.

\begin{claim} $\{C_t:t\in T_{\leq \delta}\}$ is transitive and coherent.
\end{claim}
\begin{proof} Transitivity is assured by condition (\ref{transcond2}). We check that $\ul \eta$ witnesses that $\{C_t:t\in T_{\leq \delta}\}$ is coherent. Fix $\xi<\mf c$ and $n<\oo$ such that $x_\xi \uhp n\in \dom(\varphi)$ i.e. $s=\varphi(x_\xi\uhp n)\in C_{t_\xi}$.

If $C_s$ is finite then we need that $C_s=C_{t_\xi}\cap \da s$; this is assured by condition (\ref{cond})-(d) above. Suppose that $C_s$ is infinite; we need to check that $$\eta_{t_\xi}\cap \da s \sqsubseteq \eta_s \text{ and } C_{t_\xi}\cap \da r= C_s\cap \da r $$ for $\displaystyle{r=s\uhp ht(\max_{<_T}(\eta_{t_\xi}\cap \da s))+1.}$ Recall that $$\eta_{t_\xi}=\{t_\xi\cap (\varepsilon +1):\varepsilon\in \nu_\delta\} \text{ and } \eta_s=\{s\cap (\varepsilon +1):\varepsilon\in \nu_{\max(s)}\}.$$

Note that $u\in \eta_{t_\xi}\cap \da s$ iff $u=t_\xi\cap (\varepsilon+1)$ for some $\varepsilon\in \nu_\delta\cap \max(s)$. Furthermore, $\nu_\delta\cap \max(s)=\nu_\delta\cap \varepsilon_{n-1}\sqsubseteq \nu_{\max(s)}$ by the choice of $s$ and condition (iv)-(c). Hence, as $s\sqsubseteq t_\xi$, we get that $\eta_{t_\xi}\cap \da s \sqsubseteq \eta_s$.

Finally, note that condition (\ref{cond})-(c) says that $C_s$ and $C_{t_\xi}$ agree below $s\cap (\varepsilon_{n-1}+1)$ and $\max(\eta_{t_\xi}\cap \da s)\leq (s\cap \varepsilon_{n-1}+1)$. This shows $C_{t_\xi}\cap \da r= C_s\cap \da r$.
\end{proof}

This finishes our main induction and, as transitivity and coherence are preserved at limit steps, we constructed a transitive and coherent ladder system $\ul C$ on the tree $T$. It is left to prove

\begin{claim} $Chr(X_{\ul C})>\omega$.
\end{claim}
\begin{proof} Fix a colouring $f:T\to \oo$; we will find $s,t\in T$ such that $f(s)=f(t)$ and $s\in C_t$. We fix an increasing sequence $(M_n:n\in\oo)$ of countable elementary submodels of $H(\omega_2)$  so that $S,\ul C,f \in M_n$ for all $n\in \oo$ and $\delta=M\cap \omg\in S$ for $M=\bigcup \{M_n:n\in \omega\}$.

We consider the construction of $\{C_t:t\in T_\delta\}$. There is a $\xi<\mf c$ so that $A^\xi_n=T\cap M_n$ for all $n\in \oo$ and $f_\xi=f\uhp (M\cap T)$. Our goal is to show that there is an $s\in C_{t_\xi}$ such that $f(s)=f(t_\xi)$; we let $l=f(t_\xi)$. Recall that there is an $x=x_\xi\in 2^\oo$ so that $t_\xi=\psi(x)$ and $C_{t_\xi}=\{\varphi(x\uhp k):k<\oo, x\uhp k\in \dom(\varphi)\}$; we will show that there is an $n<\oo$ such that $x\uhp n\in \dom(\varphi)$ and $f(\varphi(x\uhp n))=l$, which finishes the proof.

We first show that $C_{t_\xi}$ is infinite, equivalently that there are infinitely many $n\in \oo$ such that $x\uhp n\in \dom(\varphi)$. Suppose otherwise i.e. there is an $n_0<\oo$ such that $x\uhp n\notin \dom(\varphi)$ for all $n\in \oo\setm n_0$. Let $n$ be the minimal element of $\oo\setm n_0$ such that $l=l_n$.

Now, recall how we tried to construct $\varphi(x\uhp n)$: we looked at the set  $R^\xi_{x\uhp n}$ (elements from $A^\xi_n$ which satisfied conditions (a')-(d')) and if $R^\xi_{x\uhp n}$ was not empty then we put $x\uhp n\in \dom(\varphi)$ and chose $\varphi(x\uhp n)\in R^\xi_{x\uhp n}$. Let us show that  $R^\xi_{x\uhp n}$ is not empty. Let $$R=\{s\in T:s\geq \psi(x\uhp n), f(s)=l_n \text{ and } C_s=\{\varphi(x\uhp k):k< n, x \uhp k\in \dom(\varphi)\} \}
$$ and note that $R$ is in the model $M_n$ and $R\cap M_n\subseteq R^\xi_{x\uhp n}$. Hence, by elementarity, it suffices to show that $R\neq \emptyset$. However, this is clear as $t_\xi\in R$. This contradicts $x\uhp n\notin \dom(\varphi)$ and in turn shows that $C_{t_\xi}$ is infinite.

Now, with a quite similar argument, we will show that $x\uhp n\in \dom(\varphi)$ for the minimal $n\in \oo$ such that $l=l_n$; this finishes the proof of the claim as $s=\varphi(x\uhp n)\in C_{t_\xi}$ and $f(s)=l=f(t_\xi)$. Again, we look at $R^\xi_{x\uhp n}$ and prove that $R^\xi_{x\uhp n}\neq \emptyset$. Consider 
\begin{multline*}
R=\{s\in T\cap supp(\ul C):s\geq \psi(x\uhp n), f(s)=l_n, \\ 
\{\varphi(x\uhp k):k< n, x \uhp k\in \dom(\varphi)\}\cup\{s\} \text{  is complete},\\
 \nu_\delta\cap \varepsilon_{n-1}\sqsubseteq \nu_{\max(s)}  \text{ and }
 C_s\cap \da r= \{\varphi(x\uhp k):k< n, x\uhp k\in \dom(\varphi)\} \text{ for } r=s\cap (\vareps_{n-1}+1)\}.  
\end{multline*}
It is clear that $R\in M_n$ and $R\cap M_n\subseteq R^\xi_{x\uhp n}$ hence it suffices to show, by elementarity, that $R\neq \emptyset$. However, $t_\xi\in R$.

\end{proof}

This finishes the proof of the theorem.
\end{proof}

\section{A triangle free example}\label{tfreesec}

In this section, we adapt ideas of A. Hajnal and P. Komj\'ath \cite{half} to our setting and construct a ladder system $\ul C$ on $T=T(S)$ (with $S\subs \omg$ stationary) so that $X_{\ul C}$ is uncountably chromatic, triangle free and contains no copies of the graph $H_{\omega,\omega+2}$. A graph with these particular properties is constructed in \cite{half} but using the Continuum Hypothesis; this assumption is no longer used in our construction. 

\begin{definition} Suppose that $T$ is a tree. A cycle $x_0, x_1, .... x_n=x_0$ in $G(T)$ is \emph{special} if it is the union of two $<_T$-monotone paths. 
\end{definition}

Note that every triangle is a special cycle. Our aim is to construct a graph of the form $X_{\ul C}$ without special cycles.

\begin{definition} Suppose that $T$ is a tree of the form $T(S)$, $X$ is a subgraph of $G(T)$ and $\gamma<\omg$. We say that a  vertex $v\in T$ is \emph{$\gamma$-covered} in $X$ iff there exists a point $w\in T_{\leq \gamma}$ and a monotone path  from $w$ to $v$ in $X$.

A ladder system $\ul C$ on $T$ is \emph{sparse} iff  $s$ is not $\max(r)$-covered in $X_{\ul C}$ for each $t\in T$ and $r,s\in C_t$ with $r<s$. 
\end{definition}

Note that if $\ul C$ is sparse then $C_t$ is independent in $X_{\ul C}$ for all $t\in T$ and hence $X_{\ul C}$ is triangle free. The following was essentially proved in \cite{half} and motivates the definitions above:

\begin{lemma}\label{sparselemma} Suppose that $\ul C$ is a ladder system on $T$. Then 
\begin{enumerate}
\item if $\ul C$ is sparse then $X_{\ul C}$ contains no special cycles,
\item if $X_{\ul C}$ contains no special cycles then $X_{\ul C}$ contains no triangles or copies of $H_{\omega,\omega+2}$.
\end{enumerate}
\end{lemma}
\begin{proof} (1) Suppose that $x_0, x_1, \dots x_n=x_0$ is a special cycle with $\max(x_i)=\alpha_i$. Hence, there is $i< n$ so that $\alpha_j<\alpha_i$ if $i\neq j\leq n$. In particular, $x_{i-1}, x_{i+1}\in C_{x_i}$ and without loss of generality $\alpha_{i-1} < \alpha_{i+1} < \alpha_i$. However, this implies that $x_{i+1}$ is $\alpha_{i-1}$-covered, witnessed by the path $x_n, x_{n-1}, \dots x_{i+1}$, which contradicts that $\ul C$ is sparse.

(2) It is clear that every triangle is a special cycle. Now, suppose that $\{x_i,y_i, z,z':i<\oo\}$ is a subgraph of $X_{\ul C}$ isomorphic to $H_{\omega,\omega+2}$ i.e. the following pairs of points are edges $$\bigl \{\{x_i,y_j\},\{x_i,z\},\{x_i,z'\}:i\leq j\in \mathbb N\bigr \}.$$ First, as $x_i$ and $x_j$ has infinitely many common neighbors (for $i<j<\oo$) they must be $<_T$-comparable; hence we can suppose that $x_0<_T x_1<_T...$. Second, either $z$ or $z'$ is $<_T$-below infinitely many $x_i$ so we might as well suppose that $z<_T x_i$ for all $i<\oo$. Finally, we have $z<_T x_0<_T x_1$ and $x_0,x_1$ have infinitely many common neighbors of the form $y_j$ with $\max(y_j)>\alpha_1$. In particular, we can find special cycles of length  4 which contradicts our assumption.
\end{proof}

Hence, we will aim at constructing sparse ladder systems $\ul C$ such that the corresponding graphs $X_{\ul C}$ are uncountably chromatic. Before that, we need the following

\begin{lemma}\label{dec} Fix a stationary $S\subs \omg$ and let $T=T(S)$. Suppose that $X$ is a subgraph of $G(T)$ and $f:T\to \omega$. Then there is $\delta\in \omg$ and $t\in T_\delta$ so that for every $n\in \omega$ either:
 \begin{enumerate}
\item for every $r\geq t$ and every $\gamma\in\omg$ there is an $s\geq r$ with $f(s)=n$ which is \emph{not} $\gamma$-covered in $X$, or 
\item every $r\geq t$ with $f(r)=n$ is $\delta$-covered in $X$.
\end{enumerate}
\end{lemma}

We will say that the vertex $t$ \emph{decides $f$}.

\begin{proof} Take a countable elementary submodel $M\prec H(\omega_2)$ with $f, X, S\in M$ so that $M\cap \omg=\delta\in S$. Fix a cofinal sequence $\{\delta_n:n\in\oo\}$ of type $\oo$ in $\delta$.

Now, construct a sequence $t_0\leq.... \leq t_n\leq ...$ in $M\cap T$ so that $\max(t_n)\geq \delta_n$ and for every $n\in \omega$ either
 \begin{enumerate}[(i)]
\item for every $r \geq {t_{n+1}}$ and every $\gamma\in\omg$ there is an $s\geq r$ with $f(s)=n$ which is \emph{not} $\gamma$-covered in $X$, or 
\item there is a $\gamma_n\in \delta$ so that every $r\geq {t_{n+1}}$ with $f(r)=n$ is $\gamma_n$-covered in $X$.
\end{enumerate}
We can pick $t_0\in M\cap T$ with $\max(t_0)\geq \delta_0$ arbitrarily. Given $t_n\in M$ we select $t'_{n+1}\in M$ above ${t_n}$ so that $\max(t'_{n+1})\geq \delta_{n+1}$. If the choice $t_{n+1}=t'_{n+1}$ satisfies (i) from above then we are done; otherwise, there is $t_{n+1}\geq t_{n+1}'$ and $\gamma=\gamma_n$ so that every $r\geq {t_{n+1}}$ with $f(r)=n$ is $\gamma_n$-covered. $t_{n+1}$ and $\gamma_n$ can be chosen in $M$ by elementarity so we are done.

Now, let $t=\bigcup\{t_n:n\in\omega\}\cup\{\delta\}$ and note that $t\in T_\delta$; we claim that this $t$ decides $f$. Indeed, as $t\geq t_n$ for all $n\in \omega$ and $t_n$ satisfies (i) or (ii), $t$ must satisfy either $1.$ or $2.$ respectively.

\end{proof}

We are ready to prove the main result of this section:

\begin{theorem} \label{trianglethm} Fix a stationary $S\subseteq \omg$ and let $T=T(S)$. Then there is subgraph $X$ of $G(T)$ with $Chr(X)=\omg$ such that $X$ contains no special cycles; in particular, $X$ contains no triangles or copies of $H_{\oo,\oo+2}$.
\end{theorem}

\begin{proof} 
It suffices to construct a sparse ladder system $\ul C$ on $T$ so that $Chr(X_{\ul C})=\omg$; indeed, by Lemma \ref{sparselemma}, a sparse ladder system $\ul C$ induces a graph $X_{\ul C}$ on $T$ with no special cycles and hence no triangles or copies of $H_{\oo,\oo+2}$.

We define a sparse ladder system $(C_t:t\in T_{<\delta})$ by induction on $\delta\in S'$ and so that $C_t=\emptyset$ if $t\in T$ is a successor. Suppose we constructed $C_t$ for $t\in T_{<\delta}$ and we extend this ladder system to $T_{<\delta^+}$ where $\delta^+$ is the minimum of $S'\setm (\delta+1)$ in two steps. First, we define $C_t$ for $t\in T_\delta$ and then let $C_t=\emptyset$ for $t\in T_{<\delta^+}\setm T_{\leq \delta}$. We may suppose $\delta\in S$, otherwise $T_\delta=\emptyset$.

Let $\{(A_\xi,f_\xi,t_{0\xi}):\xi<\mf c\}$ denote a 1-1 enumeration of all triples $(A,f,t_0)$ with  $A\in [T_{<\delta}]^{\oo}$, $f:A\to \oo$ and $t_0\in A$ so that 
 \begin{enumerate}[$(\star)$]
\item for every $t\in A$ and $\vareps<\delta$ there are incomparable $s^0,s^1\in A$ so that $s^i\geq t$ and $\max(s^i)\geq \vareps$ for $i<2$, and
\end{enumerate}
By induction on $\xi<\mf c$ we define $t_\xi\in T_\delta\setm \{t_\zeta:\zeta<\xi\}$ and $C_{t_\xi}\subseteq \da t_\xi$ (while preserving that the ladder system is sparse). Suppose we have $\{t_\zeta:\zeta<\xi\}$ defined and consider the triple $(A_\xi,f_\xi,t_{0\xi})$. Fix a cofinal increasing sequence $\{\delta_n:n\in \omega\}$ of type $\oo$ in $\delta$.

We define a map $\psi:2^{<\omega}\to A_\xi$  and a partial map $\varphi:2^{<\omega}\to A_\xi$ so that 
\begin{enumerate}[(i)]
\item $\psi$ and $\varphi$ are order preserving injections and $$t_{0\xi}<\psi(x)\leq \varphi(x)\leq \psi(x\smf i)$$ for $i<2$ provided that $x\in \dom(\varphi)$,

\item \label{branchcond3} $\psi(x\smf 0)$ and $\psi(x\smf 1)$ are incomparable and contained in $A_\xi\setm T_{<\delta_n}$,
\item\label{sparsecond} $\varphi(x)$ is not $\max(\psi(x))$-covered,

\item  \underline{if} there is an $s\in A_\xi$ such that 
\begin{enumerate}
	\item $s\geq \psi(x),f_\xi(s)=n$,
	\item $s$ is not $\max(\psi(x))$-covered
\end{enumerate} \underline{then} $x\in \dom(\varphi)$ and $\varphi(x)$ satisfies (a)-(b) as well
\end{enumerate}
for all $x\in 2^n$ and $n\in \oo$.

We define $\psi(x)$ and $\varphi(x)$ for $x\in 2^n$ by induction on $n\in \omega$. We select $\psi(\emptyset)>t_{0\xi}$ arbitrarily in $A_\xi$. Given $\psi(x)$ for $x\in 2^n$ we look at the set $$R^\xi_x=\{s\in A_\xi:s\geq \psi(x),f_\xi(s)=n \text{ and } s \text{ is not } \max(\psi(x))\text{-covered}\}.$$

If $R^\xi_x$ is not empty then let $x\in \dom(\varphi)$ and pick any $\varphi(x)\in R^\xi_x$; otherwise $x\notin \dom(\varphi)$. Now, using condition ($\star$) of $A_\xi$, select incomparable $\psi(x\smf 0)$ and $\psi(x\smf 1)$ so that conditions (i)-(ii) are satisfied. This finishes the construction of $\psi$ and $\varphi$.

Now extend $\psi$ to $2^\omega$ in the obvious way:

$$\psi(x)=\bigcup\{\psi(x\uhp k):k<\oo\}\cup \{\delta\}$$

for $x\in 2^\oo$; note that $\psi(x)$ is a closed subset of $S$  by the second part of condition (\ref{branchcond3}) and hence $\psi(x) \in T_{\delta}$ for all $x\in 2^\oo$. Also, $\psi$ remains 1-1 on $2^\oo$ by the first part of condition (\ref{branchcond3}). Hence, we can find an $x_\xi \in 2^\omega$ such that $\psi(x_\xi)\in T_{\delta}\setm \{t_\zeta:\zeta<\xi\}$ and we let $t_\xi=\psi(x_\xi)$. Finally, let $$C_{t_\xi}=\bigl \{\varphi(x_\xi\uhp k):k<\oo, x_\xi\uhp k\in \dom(\varphi)\bigr \}.$$ Note that condition (\ref{sparsecond}) ensures that $C_{t_\xi}$ is sparse. This finishes the induction on $\xi<\mf c$ and in turn the induction on $\delta\in Lim(\omg)$.

We are left to prove 

\begin{claim}$Chr(X_{\ul C})>\omega$.
\end{claim}
\begin{proof} Fix a colouring $f:T\to \omega$; we will find $s,t\in T$ so that $f(s)=f(t)$ and $s\in C_t$. Take a countable elementary submodel $M\prec H(\omega_2)$ so that $S,\ul C, f\in M$ and $\delta=M\cap \omg\in S$. By Lemma \ref{dec}, we can find $t_0\in M\cap T$ so that $t_0$ decides $f$.

Now, consider the construction of $\{C_t:t\in T_\delta\}$; note that there is a $\xi<\mf c$ so that $(A_\xi,f_\xi,t_{0\xi})=(A, f\uhp A,t_0)$ where $A=T\cap M$. We will show that there is $s\in C_{t_\xi}$ with $f(s)=f(t_\xi)$. 

Let $f(t_\xi)=n$.

\begin{obs}\label{largecolour} For every $r\geq t_0$ and every $\gamma\in\omg$ there is an $s\geq r$ with $f(s)=n$ which is \emph{not} $\gamma$-covered . 
\end{obs}
\begin{proof} Recall that $t_0=t_{0\xi}$ decides $f$, so if the above statement fails then every $r\geq t_0$ with $f(r)=n$ is $\max(t_0)$-covered. In particular, $t_\xi$ is $\max(t_0)$-covered. However, this implies that $C_{t_\xi}$ is not empty and there is $s\in C_{t_\xi}$ which is $\max(t_0)$-covered (note that $s>t_0$ for all $s\in C_{t_\xi}$). However, every $s\in C_{t_\xi}$ is not $\max(t_{0\xi})$-covered by conditions (i) and (iii) above; this contradiction finishes the proof. 
\end{proof}

Recall that there is an $x\in 2^\oo$ such that $t_\xi=\psi(x)$ and $C_{t_\xi}=\{\varphi(x\uhp k):k<\oo, x\uhp k\in \dom(\varphi)\}$. Our aim is to show that $x\uhp n\in \dom(\varphi)$ and hence $f(s)=f(t_\xi)$ for $s=\varphi(x\uhp n)\in C_{t_\xi}$. Thus we need to prove that 
$$R^\xi_{x\uhp n}=\{s\in A_\xi:s\geq \psi(x\uhp n),f_\xi(s)=n \text{ and } s \text{ is not } \max(\psi(x\uhp n))\text{-covered}\}$$ is not empty.

Let $$R=\{s\in T:s\geq \psi(x\uhp n),f(s)=n \text{ and } s \text{ is not } \max(\psi(x\uhp n))\text{-covered}\}$$ and note that $R^\xi_{x\uhp n}=R\cap M$ and $R\in M$.  Hence, by elementarity, it suffices to show that $R\neq \emptyset$. This clearly follows from Observation \ref{largecolour} applied to $r=\psi(x\uhp n)$ and $\gamma=\max(\psi(x\uhp n))$.

\end{proof}
This finishes the proof of the theorem.
\end{proof}

Let us remark that sparse and transitive ladder systems represent two extremes in the spectrum of subgraphs of $G(T)$; if $\ul C$ is sparse then $C_t$ is an independent set while if $\ul C$ is transitive then $C_t$ is a complete subgraph.

\section{Remarks}\label{remarkssec}

We would like to point out that some of the graphs defined in our paper satisfy strong \emph{partition properties}. If $T$ is a tree and $X$ is a subgraph of $G(T)$ then we write  $$X\to (K_{\omega+1})^1_\omega$$ iff for every colouring $f:T\to \oo$ there is an $n\in \oo$ and a set $A\subseteq T\cap f^{-1}(n)$ of $<_T$-order type $\oo+1$ such that $A$ spans a complete graph (i.e. $A$ is a monochromatic copy of $K_{\oo+1})$. Clearly, $X\to (K_{\omega+1})^1_\omega$ implies that $Chr(X)>\oo$ but not necessarily the other way; indeed, as seen in Theorem \ref{trianglethm}, there are even triangle free subgraphs of $G(T)$ (for some $T$) which are uncountably chromatic.

Let us first show that satisfying the above partition property or having large chromatic number are equivalent for \emph{transitive} ladder systems. 

\begin{prop}Suppose that $T$ is a tree and $\ul C$ is a ladder system on $T$. If $\ul C$ is transitive and $\chr(X_{\ul C})>\oo$ then $$X_{\ul C}\to (K_{\omega+1})^1_\omega.$$
\end{prop}
\begin{proof} 
Fix an $f:T\to \omega$; we will show that there is an $n\in\omega$ and $t\in f^{-1}(n)$ so that $A=C_t\cap f^{-1}(n)$ is infinite hence, by transitivity, $A\cup \{t\}$ gives a monochromatic copy of $K_{\oo+1}$ in $X_{\ul C}$. 

Suppose otherwise i.e. $C_t\cap f^{-1}(n)$ is finite for every $t\in T$ with $f(t)=n$. We can define a new colouring $g:T\to \oo\times \oo$ using induction on the height so that $g(t)=(f(t),g_1(t))$ where $g_1(t)=\max\{g_1(s):s\in C_t\cap f^{-1}(n)\}+1$ with $n=f(t)$. It is easy to see that $g$ witnesses $Chr(X_{\ul C})\leq \oo$ which is a contradiction.
\end{proof}

The above proposition is nicely complemented by 

\begin{obs}If $T$ is a tree of height $\omg$ and $\ul C$ is a ladder system on $T$ then there is no complete subgraph in $X_{\ul C}$ of $<_T$-order type $\omega+2$.
\end{obs}

Recall that in the proof of Theorem \ref{mainthm2}, we used a ladder system on $\omg$ (denoted by $\ul \nu$ there) to define another ladder system (denoted by $\ul \eta$) on $T=T(S)$ for $S\subseteq \omg$ in a very natural way. We did not consider the subgraph of $G(T)$ corresponding to $\ul \eta$ at that point so let us present a result here.

\begin{prop}Suppose that $S\subseteq \omg$ is stationary and let $T=T(S)$. Fix a ladder system $\ul \nu=\{\nu_\delta:\delta\in Lim(\omg)\}$ on $\omg$. Let $$C_t=\{t\cap (\vareps+1) :\vareps\in \nu_\delta\}$$ for any limit $t\in T_\delta$ and $\delta\in S$ and let $C_t=\emptyset$ otherwise. Then $$X_{\ul C}\to (K_{\omega+1})^1_\omega.$$
\end{prop}
\begin{proof} Let $f:T\to \omega$. We say that $D\subseteq T$ is \emph{dense above} $t\in T$ iff for every $s\geq t$ there is $r\in D$ such that $r\geq s$. $D$ is \emph{empty above} $t$ iff $s\notin D$ for every $s\geq t$.

\begin{claim}
There is $t_0$ in $T$ so that either $f^{-1}(k)$ is empty or dense above $t_0$ for every $k\in \omega$ .
\end{claim}
\begin{proof} The proof is very similar to the argument seen in Lemma \ref{dec} so we will be brief here. Take a countable elementary submodel  $M\prec H(\omega_2)$ with $f,S\in M$ so that $\delta=M\cap \omg\in S$. Build a sequence $s_0\leq s_1\leq \dots$ in $T\cap M$ so that $(\max(s_k):k\in \oo)$ is a cofinal $\oo$-type sequence in $\delta$ and $f^{-1}(k)$ is either empty or dense above $s_k$. It is easy to see that $$t_0=\bigcup\{s_k:k\in\oo\}\cup\{\delta\}$$ is in $T$ and satisfies the claim.
\end{proof}

Take a countable elementary submodel $M\prec H(\omega_2)$ with $t_0,f,S, \ul \nu\in M$ and $\delta=M\cap \omg\in S$.  Let $\{k_n:n\in\omega\}$ enumerate those $k\in \omega$ so that $f^{-1}(k)$ is dense above $t_0$ (or equivalently, not empty above $t_0$), each $\omega$ times. Let $(\delta_n:n\in\oo)$ be an arbitrary cofinal $\oo$-type sequence in $\delta$.

We construct $t_0\leq t_1\leq ....\leq t_n \leq ...$ in $T\cap M$ so that 

\begin{enumerate}
\item $\max(t_{n+1})\geq \delta_n$,
\item $t_{n+1}\cap (\vareps_n+1)=t_n$ for $\vareps_n=\min \nu_\delta\setm (\max(t_n)+1)$,
\item \ul{if} there is $t\geq t_n$ in $T\cap M$ such that 
\begin{enumerate}
	\item $\max(t)\geq \delta_n$, $t\cap (\vareps_n+1)=t_n$,
	\item $\{t_i:i\leq n\}\subseteq C_t$ and
	\item $f(t)=k_n$ 
\end{enumerate} 
 \ul{then} $t=t_{n+1}$ satisfies (a)-(c) as well.
\end{enumerate}
Let $t=\bigcup\{t_n:n\in\omega\}\cup \{\delta\}$; note that $t\in T_\delta$ and $t\geq t_n$ for all $n\in \omega$. Also, (2) ensures that $t_n\in C_t$ for $n\in \oo$ as $t_n=t\cap (\vareps_n+1)$ and $\vareps_n\in\nu_\delta$.

Let $k=f(t)$; as $t\geq t_0$ we know that $f^{-1}(k)$ is dense above $t_0$. We claim that $$A=\{t_{n+1}:k=k_n, n\in \omega\}\cup \{t\}$$ is a complete subgraph $X_{\ul C}$ and $A$ is coloured with $k$ which finishes the proof. 
It suffice to show that whenever $k_n=k$ and $t_{n+1}$ is constructed then condition (3) is satisfied. Fix an $n\in \oo$ such that $k_n=k$. Consider the set $$R_n=\{t\in T\cap M:\max(t)\geq \delta_n, t\cap (\vareps_n+1)=t_n,
	\{t_i:i\leq n\}\subseteq C_t \text{ and } f(t)=k_n \}$$ and we wish to show that $R_n\neq \emptyset$.
	
	It suffices to show that $$R=\{t\in T:\max(t)\geq \delta_n, t\cap (\vareps_n+1)=t_n,
	\{t_i:i\leq n\}\subseteq C_t \text{ and } f(t)=k_n \}$$ is not empty as $R_n=R\cap M$ and $R\in M$. However $t\in R$ which finishes the proof.

\end{proof}

Finally, we mention that the following most general form of the Erd\H os-Hajnal problem is still open:

\begin{prob} Does every uncountably chromatic (or $\omg$-chromatic) graph contain an $\oo$-connected subset?
\end{prob} 

We know, by Theorem \ref{mainthm}, that this $\oo$-connected set can only be countable in some cases however excluding countable $\oo$-connected subsets seems to be a very hard problem.

\section{Acknowledgments}

I would like to thank Chris Eagle for enlightening conversations in the early stages of this project, as well as William Weiss for his help in clarifying my arguments.

\end{document}